\documentclass[12pt,reqno]{amsart}

\usepackage{latexsym}
\usepackage{graphicx}
\newcommand{\QQ}{\mathcal{Q}}
\newcommand{\T}{{\mathbf T}^m}

\newcommand{\szego}{Szeg\"o }

\newcommand{\Si}{\Sigma}

\newcommand{\kahler}{K\"ahler }
\newcommand{\poly}{{\mathcal Poly}}

\newcommand{\wh}{\widehat}
\newcommand{\PP}{{\mathbb P}}
\newcommand{\N}{{\mathbb N}}
\newcommand{\R}{{\mathbb R}}
\newcommand{\C}{{\mathbb C}}

\newcommand{\Z}{{\mathbb Z}}

\newcommand{\CP}{\C\PP}
\renewcommand{\d}{\partial}
\newcommand{\dbar}{\bar\partial}
\newcommand{\ddbar}{\partial\dbar}

\newcommand{\half}{{\frac{1}{2}}}

\newcommand{\FS}{{{\operatorname{FS}}}}

\renewcommand{\phi}{\varphi}

\newcommand{\ccal}{\mathcal{C}}

\newcommand{\ecal}{\mathcal{E}}

\newcommand{\hcal}{\mathcal{H}}

\newcommand{\lcal}{\mathcal{L}}

\newcommand{\ncal}{\mathcal{N}}
\newcommand{\ocal}{\mathcal{O}}
\newcommand{\pcal}{\mathcal{P}}

\newcommand{\rcal}{\mathcal{R}}

\newcommand{\al}{\alpha}

\newcommand{\om}{\omega}

\newtheorem{theo}{{\sc Theorem}}[section]
\newtheorem{maintheo}{{\sc Theorem}}
\newtheorem{cor}[maintheo]{{\sc Corollary}}

\newtheorem{lem}[theo]{{\sc Lemma}}
\newtheorem{prop}[theo]{{\sc Proposition}}

\newenvironment{rem}{\medskip\noindent{\it Remarks:\/} }{\medskip}
\newenvironment{defin}{\medskip\noindent{\it Definition:\/} }{\medskip}

\title[Bernstein polynomials, Bergman kernels and toric K\"ahler  varieties]
{Bernstein polynomials, Bergman kernels and toric K\"ahler
varieties}

\author{Steve Zelditch }
\address{Department of Mathematics, Johns Hopkins University, Baltimore, MD
21218, USA}
\email{zelditch@math.jhu.edu}

\thanks{Research partially supported by NSF grant  DMS-0603850}

\date{\today}

\begin{document}

\maketitle

\begin{abstract} We show that the  classical Bernstein polynomials $B_N(f)(x)$  on the interval $[0, 1]$ (and their higher
dimensional generalizations on  the simplex $\Sigma_m \subset
\R^m$) may be expressed in terms of Bergman kernels for the
Fubini-Study metric on $\CP^m$: $B_N(f)(x)$ is obtained by
applying   the Toeplitz operator $f(N^{-1} D_{\theta})$  to the
Fubini-Study Bergman kernels.  The expression generalizes
immediately to any toric \kahler variety and Delzant polytope, and
gives a novel definition of Bernstein `polynomials' $B_{h^N}(f)$
relative to any toric \kahler variety. They uniformly approximate
any continuous function $f$ on the associated polytope $P$  with
all the properties of classical Bernstein polynomials. Upon
integration over the polytope one obtains a complete asymptotic
expansion for the Dedekind-Riemann sums $\frac{1}{N^m}
\sum_{\alpha \in N P} f(\frac{\alpha}{N}) $ of $f \in
C^{\infty}(\R^m)$, of a type similar to the Euler-MacLaurin
formulae.
\end{abstract}

\section*{Introduction}

Our starting point is the  observation that  the classical
Bernstein polynomials
\begin{equation} \label{NDBERNDEF} B_N(f)(x) =  \sum_{\alpha \in \N^m: |\alpha| \leq N}{N \choose
\alpha} x^{\alpha}(1 - ||x||)^{N - |\alpha|} f(\frac{\alpha}{N}),
\end{equation}   on the $m$-simplex $\Sigma_m \subset \R^m$ may be expressed   in terms of the  Bergman-\szego
 kernels $\Pi_{h_{FS}^N}(z,w)$ for the
Fubini-Study metric on $\CP^m$:  Let $e^{i \theta}$ denote the
standard $\T = (S^1)^m$ action on $\C^m$ and and let
$D_{\theta_j}$ denote the linearization (or `quantization') of its
infinitesimal generators on $H^0(\CP^m, \ocal(N))$. As will be
shown in \S \ref{FS} (see also \S \ref{BPTO}),
\begin{equation} \label{BN} B_N(f)(x)  =  \frac{1}{\Pi_{h_{FS}^N}(z,z)} f(N^{-1} D_{\theta}) \Pi_{h_{FS}^N}(e^{i \theta} z,
 z)
|_{\theta = 0, z = \mu_{h_{FS}}^{-1}(x)}, \end{equation} where $f
\in C_0^{\infty}(\R^m)$. Here, $\Pi_{h_{FS}^N}$  denotes the
\szego or Bergman kernel on  powers $\ocal(N) \to \CP^m$ of the
invariant hyperplane line bundle,  $f(N^{-1} D_{\theta})$ is
defined by the spectral theorem and $\mu_{h_{FS}}$ is the moment
map corresponding to $h_{FS}$. Thus, the Bernstein polynomial $B_N
f(x)$ is the Berezin lower symbol for the Toeplitz operator $\Pi_N
f(N^{-1} D_{\theta}) \Pi_N$, i.e. the value of its kernel on the
diagonal. From this formula, many properties of Bernstein
polynomials may be derived from properties of the Fubini-Study
Bergman-\szego kernel.

 Furthermore, the  formula
(\ref{BN})   generalizes immediately to any polarized  toric
\kahler variety $(L, M, \omega)$ and defines analogues
$B_{h^N}(f)(x)$  of Bernstein polynomials for any Delzant  poytope
$P$ and any positively curved toric hermitian metric $h$   on the
invariant line bundle associated to $P$. We simply replace the
Hermitian line bundle $\ocal(1) \to \CP^m$ with its Fubini-Study
metric  by any toric invariant Hermitian line bundle $(L, h) \to
(M, \omega)$ (see Definition \ref{BB}).

The connection between Bernstein polynomials and Bergman-\szego
kernels may be used to obtain asymptotic expansions of Bernstein
polynomials as the degree $N \to \infty$;

\begin{theo} \label{BBa} Let $(L, h) \to (M, \omega)$ be a toric Hermitian invariant
line bundle over a  toric \kahler variety with associated moment
polytope $P$. Let $f \in C_0^{\infty}(\R^m)$ and let
$B_{h^N}(f)(x)$ denote its Bernstein polynomial approximation in
the sense of Definition \ref{BB}.  Then there exists a complete
asymptotic expansion,
$$ B_{h^N}(f)(x) = f(x) + \lcal_1 f(x) N^{-1} + \lcal_2 f(x) N^{-2} + \cdots + \lcal_m f(x) N^{-m}
+ O(N^{- m - 1}), $$ in $C^{\infty}(\bar{P})$, where $\lcal_j$ is
a differential operator of order $2j$ depending only on curvature
invariants of the metric $h$; the expansion may be differentiated
any number of times.
\end{theo}
In the case of classical Bernstein polynomials (\ref{NDBERNDEF})
(i.e. the interval or simplex) , this expansion has recently  been
derived by L. H\"ormander  \cite{Ho} by a different method (see
\ref{ASYMPTOTICS}). The approach taken here is to use the Boutet
de Monvel-Sj\"ostrand approximations of Bergman-\szego kernels,
with some simplifications in the case of  toric hermitian metrics
\cite{BSj,STZ}. The operators $\lcal_j$ are computable from the
coefficients of the asymptotic expansion of the Bergman-\szego
kernel $\Pi_{h^N}(z,z)$ on the diagonal in \cite{Z2, Lu}. It
should be noted that for general toric Hermitian line bundles, the
Bernstein `polynomials' are not quite polynomials in the usual
sense, although they are algebro-geometric objects in the sense of
\cite{D,T}; see \S \ref{DEF} for further discussion.

As defined in (\ref{BN}) and in  Definition \ref{BB}, the
Bernstein polynomials are quotients \begin{equation}
\label{QUOTIENT} B_{h^N}(f)(x) = \frac{\ncal_{h^N}
f(x)}{\Pi_{h^N}(\mu_h^{-1}(x), \mu_h^{-1}(x))}
\end{equation} of a {\it numerator polynomial} $\ncal_{h^N} f(x)$
by  the denominator $\Pi_{h^N}(z,z)$ with $\mu_h(z) = x$. Here,
$\mu_h$ is the moment map associated to the \kahler form
$\omega_h$ associated to $h$. The numerator polynomials also admit
complete asymptotic expansions, and indeed the Bernstein
polynomial expansions are derived from the numerator expansion and
from the asymptotic expansion of the denominator. Hence,  Theorem
\ref{BBa} follows from:

\begin{theo} \label{NUM} With the same assumptions as above,  there
exist differential operators $\ncal_j$,   such that
$$ \ncal_{h^N}(f)(x)
\sim  N^m   f(x)  + N^{m-1} \ncal_1 f(x)  + \cdots,$$ where the
operators $\ncal_j$ are computable from the Bergman kernel
expansion for $\Pi_{h^N}(z,z)$.
\end{theo}

Theorem \ref{NUM}  has an application to Dedekind-Riemann sums
over lattice points in dilates of the polytope $P$, i.e. sums of
the form
$$ \sum_{\alpha \in N P } f(\frac{\alpha}{N}),
\;\; f \in C_0^{\infty}(\R^m). $$
 Upon integration of $\ncal_{h^N} f(x)$
over $P$ one obtains:

\begin{cor} \label{EM} Let  $f \in C_0^{\infty}(\R^m)$. Then there
exist differential operators $\ecal_j$,   such that
$$ \sum_{\alpha \in N P } f(\frac{\alpha}{N})
\sim  N^m  \int_P f(x) dx + \frac{ N^{m-1}}{2}  \int_{\partial P}
f(x) d \sigma + N^{m-2} \int_{P} \ecal_2 f (x) dx + \cdots,$$
where $\sigma$ is the Leray measure on $\partial P$ corresponding
to the affine defining functions $\ell_r (x) = \langle x, \nu_r
\rangle $ of the boundary facts (cf. \ref{ELLDEF}). That is, on
the $r$th facet of $\partial P$, $d \ell_r \wedge d\sigma = dx$.
\end{cor}
\noindent  Exact and asymptotic formulae for $ \sum_{\alpha \in N
P} f(\frac{\alpha}{N})$  have been previously proved for special
$f$ using the generalized Euler-MacLaurin formulae of
Khovanskii-Pukhlikov, Brion-Vergne, Guillemin-Sternberg and others
(cf. \cite{G,GS,GSW,KSW}).  For purposes of comparison, Theorem
4.2 of \cite{GS} states that for $f \in C_0^{\infty}(\R^n)$,
\begin{equation} \label{EMG} \sum_{\alpha \in \N^m: |\alpha| \leq
N} f(\frac{\alpha}{N}) \sim \left( \sum_F \sum_{\gamma \in
\Gamma^1_F} \tau_{\gamma} (\frac{1}{N} \frac{\partial}{\partial
h}) \int_{P_h} f(x) dx \right) |_{h = 0} + O(N^{- \infty}),
\end{equation} where the sums involve various data associated to
the polytope $P$ and where $P_h$ is a parallel dilate of $P$.  We
refer to \cite{GS} for the notation. The two term expansion given
in Corollary \ref{EM}  was stated  in \cite{Sz}. It is
straightforward to generalize the formula and proof to the case
where  $f$ is a symbol as in \cite{GSW}, and to obtain remainder
estimates in the expansion.

 A significant difference between the Euler-MacLaurin  and
 the Bernstein methods for obtaining expansions of Dedekind-Riemann sums
 $ \sum_{\alpha \in \N^m: |\alpha|
\leq N} f(\frac{\alpha}{N})$  is that the Bernstein approaches
uses an arbitrary toric \kahler metric while the  Euler-MacLaurin
approach is metric independent. This reflects the fact that the
Bernstein approach is to integrate the pointwise expansion of
Theorem \ref{NUM}, which depends on the metric $h$. The metric
independence of the expansion in Corollary \ref{EM} is equivalent
to a sequence of integration by parts identities involving
curvature invariants. For instance, we obtain the second term in
the  expansion  in \S \ref{EMa} by using an integration by parts
identity on polytopes due to Donaldson \cite{D2}; see also \S
\ref{FS} for the simplest case. Conversely, comparision   of the
metric expansion in Theorem \ref{NUM} and the Euler-MacLaurin
expansion in  (\ref{EMG}) gives another proof of this identity,
and generates further identities in the lower order terms
 for any choice of toric hermitian metric.

The connection between Bernstein polynomials, Bergman kernels and
Berezin symbols appears to be new, and  one of the principal
motivations of this article is simply to point out the  toric
geometry underlying the classical Bernstein polynomials. But a
further motivation is that the generalized Bernstein polynomials
should be useful in the program of Yau-Tian-Donaldson of making
algebro-geometric (i.e. polynomial) approximations to
transcendental geometric objects on \kahler varieties (cf.
\cite{D1,T}). For instance, in
  \cite{SoZ,SoZ2} what we recognize in this article as Bernstein   polynomials were
  used to  approximate geodesic rays in $C^2$ (see also \cite{PS}).
 However,   the function
  $f$ in that paper also depended on $N$ in a subtle way and so
  the polynomials were much more complicated than the Bernstein
  polynomials of this article.  The  article
  \cite{Ho}  also concerns
relations between  Bernstein polynomials and Bergman kernels, but
mainly for the opposite purpose of deriving Bergman kernel
expansions on Reinhardt domains from classical Bernstein
polynomial expansions on the simplex. The exposition in \S
\ref{BKBE} was influenced by its analysis of Bernstein
polynomials. It also draws on some of the analysis of \cite{SoZ}.

In addition to the Bergman-toric generalization of Bernstein
polynomials, there also exists a  probabilistic generalization of
Bernstein polynomial which replaces ${N \choose \alpha}$ by the
weighted number of lattice paths from $0$ to $\alpha$ with steps
in the polytope $P$. This definition also  coincides with the
canonical one in the case of the Fubini-Study metrics on $\CP^m$
but in general gives a different class of polynomials defined on
the simplex of  probability measures on $\{1, \dots, m\}$. In the
case of the simplex $\Sigma_m = P$, both spaces are the same, but
in general they are not. The relevant analysis could be obtained
form \cite{TZ}; we  will not discuss these generalizations here.

We would like to thank H. Hezari for a careful reading of the
article and for pointing out some notational inconsistencies and
misprints  in an earlier version.

\section{\label{FS} Fubini-Study and classical Bernstein polynomials}

Let  us begin by explaining in more detail the Bernstein-Bergman
connection for the Fubini-Study metric in one complex dimension.
We recall that Bernstein polynomials of one variable give
canonical uniform polynomial approximations to continuous
functions  $f \in C([0,1])$:
\begin{equation} \label{1DBERNDEF} B_N(f)(x) = \sum_{j = 0}^N {N \choose j} f(\frac{j}{N}) x^j
(1 - x)^{N-j}.  \end{equation}  They have the special feature that
they simultaneously uniformly approximate all derivatives of $f$
if $f \in C^k$, i.e.  $B_N(f)^{(k)}(x) \to f^{(k)}(x)$ (cf.
\cite{L}),  and  if $f \in C^{\infty}$ there exists a complete
asymptotic expansion (\cite{Ho})
\begin{equation} \label{ASYMPTOTICS} B_N(f)(x) \sim \sum_{\mu = 0}^{\infty} L_{\mu}(x,
\frac{d}{dx}) f(x) N^{- \mu} \end{equation}  for certain
polynomial differential operators $L_{\mu}(x, \frac{d}{dx})$,
$$L_0 = 1, \; L_1 = \frac{1}{2}(x - x^2) \frac{d^2}{dx^2}, \;\;
L_2 = \frac{1}{6} (x - x^2) (1 - 2x) \frac{d^3}{dx^3} +
\frac{1}{8} (x - x^2)^2 \frac{d^4}{dx^4}. $$ In this case, $B_N(f)
= \frac{1}{N + 1} \ncal_N(f)$ (cf. Theorem  \ref{NUM}), and also
$${N \choose j} \int_0^1  x^j
(1 - x)^{N-j} dx = {N \choose j} \frac{j! (N - j)!}{(N + 1)!} =
\frac{1}{N + 1}. $$ Hence, (\ref{1DBERNDEF}) implies that
\begin{equation} \label{1DIM}\begin{array}{lll}  \int_0^1 \ncal_N(f)(x) dx & = &  \sum_{j = 0}^N f(\frac{j}{N})  \\ && \\
&& = (N + 1) \left( \int_0^1 f(x) dx + \frac{1}{2 N} \int_0^1 (x -
x^2) f''(x) dx + \cdots \right) \\ && \\
&& = (N + 1) \left( \int_0^1 f(x) dx + \frac{1}{2 N} \left( f(1) -
f(0) - 2 \int_0^1 f(x) dx \right)+ \cdots \right)
\\ && \\
&& = N \int_0^1 f(x) dx + \frac{1}{2 } (f(1) - f(0)) +
O(\frac{1}{N}).
\end{array} \end{equation}
We included the routine details to point out that obtaining the
first two terms of the Euler-MacLaurin Riemann sum expansion in
Theorem  \ref{EM} required two integrations by parts and
cancellations of $\int_0^1 f(x) dx$ in the constant term between
the subleading term of the dimension (Riemann-Roch) polynomial ($N
+ 1$) term and in the  $\int_0^1 L_1 f(x) dx$ term. Similar
cancellations occur in the general case (see the proof of Theorem
\ref{EM}).

We now relate the  Bernstein polynomials $B_N(f)$ on $[0, 1]$ to
the Bergman kernel for the Fubini-Study metric on $\CP^1$. The
discussion is almost the same for the $m$-simplex $\Sigma_m
\subset \R^m$ and the Bergman kernel for  the Fubini metric on
$\CP^m$, so we carry it out in all dimensions. We first need to
recall some standard facts about the Bergman or \szego kernels for
the Fubini-Study metric.

By the $m$-simplex we mean the convex set $\Sigma_m = \{(x_1,
\dots, x_m) \in \R_+^m: ||x|| : = \sum_{j = 1}^m x_j \leq 1\}. $
We denote its dilate by $N \in \N$ by $N \Sigma_m$.
 As discussed in \cite{STZ}
and elsewhere (see \cite{STZ} for references), the space
 $\poly(N \Si_m)$ of polynomials with exponents $\alpha \in N \Sigma_m$   can be identified
with the space of degree-$N$ homogeneous holomorphic polynomials
 in $m+1$ variables by
identifying  the (non-homogeneous) polynomial
$$f(z_1,\dots,z_m)= \sum _{|\al|\le N} c_\al
z^\al \qquad (z^\al=z_1^{\al_1} \cdots z_m^{\al_m})$$ with the
homogeneous polynomial
$$F(\zeta_0, \dots,
\zeta_m) = \sum_{|\al| \le N} c_\al
\zeta_0^{N-|\al|}\zeta_1^{\al_1} \cdots \zeta_m^{\al_m}\;.$$ The
space $\poly(N\Si_m)$ has a natural  $\lcal^2$ inner product,
\begin{equation}\label{IP}\langle f, \bar g \rangle = \frac 1{m!}\int _{S^{2m+1}}
F\overline{G} \,d\nu,\end{equation}

 This inner product is
equivalent to viewing $f, g$ as a holomorphic sections of the
$N$th power $\ocal(N)$ of the hyperplane  line bundle $\ocal(1)
\to \CP^m$ dual to the tautological line bundle.  The line bundle
$\ocal(1)$ carries a natural metric $h_\FS$ given by
\begin{equation}\label{hfs} \|s\|_{h_\FS}([w])=\frac{|(s,w)|}{|w|}\;,
\quad\quad w=(w_0,\dots,w_m)\in\C^{m+1}\;,\end{equation} for
$s\in\C^{m+1*}\equiv H^0(\C\PP^m,\ocal(1))$, where
$|w|^2=\sum_{j=0}^m |w_j|^2$ and $[w]\in\C\PP^m$ denotes the
complex line through $w$. The \kahler form on $\CP^m$ is the
Fubini-Study form
\begin{equation}
\omega_\FS=\frac{\sqrt{-1}}{2}\Theta_{h_\FS}=\frac{\sqrt{-1}}{2}
\ddbar \log |w|^2 \,.\end{equation} The natural Fubini-Study inner
product on sections is then
$$\langle s_1, s_2 \rangle = \int_{\CP^m} (s_1, s_2)_{h_{FS}}
\omega_{FS}^m/m!. $$ In an affine chart and local frame $e$,
sections have the form $f e$ where $f$ is a polynomial and the
inner product takes the explicit form
\begin{equation}  \langle f, \bar g \rangle =  \frac{1}{m!} \int
_{\C^m}\frac{f(z)\overline{g(z)}}{(1+\|z\|^2)^N}
\,\om_\FS^m(z),\quad f,g\in  \poly(N\Si_m).  \end{equation} Both
versions of the inner product generalize to any holomorphic line
bundle.

A basis for $\poly(N\Si_m)$ is given by  the monomials
$\chi_\al(z) = z_1^{\alpha_1} \cdots z_m^{\alpha_m}$, $|\al|\le
N$.  The monomials $\{ \chi_{\alpha}\}$ are orthogonal but not
normalized. Their $\lcal^2$ norms given by the inner product
(\ref{IP}) are:
\begin{equation}\label{IP2}\|\chi_\al\|
=\left[\frac{(N-|\al|)!\al_1!
\cdots\al_m!}{(N+m)!}\right]^\half\;.
\end{equation}   Thus,  an orthonormal basis for
$\poly(N\Si_m)$ is  given by the monomials
\begin{equation}\label{normchi} \frac{1}{\|\chi_\al\|}\,\chi_\al=
\left[\frac{(N+m)!}{(N-|\al|)!\al_1!\cdots\al_m!}\right]^\half
\chi_\al= \sqrt{\frac{(N+m)!}{N!} {N\choose \al}}\ \chi_\al \
,\qquad |\al|\le N\;.\end{equation} where
\begin{equation}\label{multinom}
{N\choose\al}= \frac{N!}{(N-|\al|)!\alpha_1!\cdots \alpha_m!}\;.
\end{equation}
 We let
$\wh\chi_\al^N:S^{2m+1}\to \C$ denote the homogenization of
$\chi_\al$:
\begin{equation}\label{chihat}\wh\chi_\al^N(x) =  x_0^{N-|\al|}x_1^{\al_1}\cdots
x_m^{\al_m} \;.\end{equation}

The Bergman or   \szego kernel $\Pi_{h_{FS}^N}$ for the
Fubini-Study metric is   the orthogonal projection to the space
$H^0(\CP^m, \ocal(N))$ of holomorphic sections with respect to the
inner produced induced by $h_{FS}$, which lifts to the
 orthogonal projection to $\poly(N\Si)$. It is thus given by
\begin{equation}\label{szego-proj}\Pi_{N}(x,y)
=\sum_{|\al|\le N}\frac{1}{\|\chi_\al\|^2}\wh \chi_\al(x)
\overline {\wh \chi_\al(y)}  =\frac{(N+m)!}{N!}\langle x,\bar
y\rangle^N \;,\end{equation} for $x,y\in S^{2m+1}$. In particular,
on the diagonal we have $\langle x, x \rangle = 1$ and
\begin{equation}\label{szego-projD}\Pi_{N}(x,x)
 =\frac{(N+m)!}{N!}.\end{equation}

In terms of the standard local affine frame on $\C^m$, we have
$\wh \chi^N_\al(z)= \frac{z^\al}{(1+\|z\|^2)^{p/2}}\;$, and hence
\begin{eqnarray}\Pi_{h_{FS}^N}(z,w)
=\frac{(N+m)!}{N!} \frac{\sum_{|\alpha|\le N }{N\choose\al}z^\al
\bar w^\al}{(1+\|z\|^2)^{N/2} (1+\|w\|^2)^{N/2}}
\label{Sz} \end{eqnarray}

We now have the ingredients to identify Bernstein polynomials for
the simplex $N \Sigma_m$ in terms of the Fubini-Study
Bergman-\szego kernel. The  \kahler potential of the Fubini-Study
metric is  $\phi_{FS}= \log (1 + ||z||^2)$ where  $||z||^2 =
\sum_j |z_j|^2$, and its moment map is
$$\mu_{h_{FS}}(z) = (\frac{|z_1|^2}{1 + ||z||^2},\dots,
\frac{|z_m|^2}{(1 + ||z||^2}). $$  The Fubini-Study  symplectic
potential is the convex function on $\Sigma_m$ given by the
Legendre transform of $\phi_{FS}$ in logarithm coordinates,
$$u_0(x) = \sum_{j= 1}^m x_j \log x_j + (1 - ||x||) \log (1 -
||x||)$$ where $||x|| = \sum_{j = 1}^m x_j. $  A simple
calculation shows that  the Bernstein terms may be expressed in
terms of the symplectic potential as
\begin{equation} \label{NCHOOSE} {N \choose \alpha} x^{\alpha} (1
- ||x||)^{N - |\alpha|} = \frac{N!}{(N+m)!} \frac{e^{N \left(
u_0(x) + \langle \frac{\alpha}{N} - x, \nabla u_0(x) \rangle
\right)}}{||z^{\alpha}||^2_{h_{FS}^N}} .
\end{equation} It follows that \begin{equation} \label{FSa}
\begin{array}{lll} B_N(f)(x) && = \frac{1}{\Pi_{h_{FS}^N}(z, z)} \sum_{\alpha = 0}^N
f(\frac{\alpha}{N}) \;  \frac{e^{N \left( u_0(x) + \langle
 \frac{\alpha}{N} - x, \nabla u_0(x) \rangle
\right)}}{||z^{\alpha}||^2_{h_{FS}^N}}, \;\; z =
\mu_{h_{FS}}^{-1}(x).
\end{array} \end{equation} On the other hand, one can also express
the Bergman-\szego kernel in terms of the symplectic potential at
the points $(e^{i \theta}z, z)$ as
\begin{equation} \label{BERNBERGFS} \begin{array}{lll} \Pi_{h_{FS}^N}(e^{i \theta} z,z) & = &  \sum_{\alpha =
0}^N e^{i \theta \alpha}
 \frac{e^{N \left( u_0(x) +
\langle \frac{\alpha}{N} - x, \nabla u_0(x) \rangle
\right)}}{||z^{\alpha}||^2_{h_{FS}^N}} \\ && \\
& = &    \Pi_{h_{FS}^N}(z,  z) \sum_{\alpha = 0}^N  {N \choose
\alpha} e^{i \theta \alpha} x^{\alpha} (1 - ||x||)^{N - |\alpha|}.
\end{array} \end{equation}
Indeed,  comparing (\ref{Sz}) and (\ref{BERNBERGFS}), we see that
the two expressions for the Bergman-\szego kernel agree as long as
\begin{equation} \label{AGREE} |z^{\alpha}|^2 e^{- N \log (1 + ||z||^2)}  = e^{N \left(
u_0(x) + \langle \frac{\alpha}{N} - x, \nabla u_0(x) \rangle
\right)}, \;\; \mbox{when} \; \mu_{h_{FS}}(z) = x,
\end{equation}
and this follows from the pair of identities,
$$|z^{\alpha}|^2 = e^{\langle \alpha, \nabla u_0(x) \rangle}, \;\;
\log (1 + |z|^2) =  \langle x, \nabla u_0(x) \rangle - u_0(x)\;\;
\mbox{when} \; \mu_{h_{FS}}(z) = x.
$$ On the open orbit, we may use
logarithmic coordinates $z = e^{\rho/2 + i \theta}$. Then $\rho =
\nabla u_0(x)$ and the identities are  equivalent to the fact that
the \kahler potential and symplectic potential are Legendre
transforms of each other. Since both sides of (\ref{BERNBERGFS})
are continuous, the equality extends to all of $M$ and $\bar{P}$.

Applying the operator $f(\frac{D_{\theta}}{N})$ just replaces
$e^{i \theta \alpha}$ by $f(\frac{\alpha}{N})$. Then, dividing by
$\Pi_{h_{FS}^N}(z, z)$ gives (\ref{FSa}) and (\ref{BN}). Together
with the
 formulae above for norms of monomials and the \szego kernel in
 dimension $m$, the formula (\ref{NDBERNDEF}) also reduces to
 (\ref{FSa}).

\section{\label{DEF}Definition of the generalized  Bernstein polynomials}

We now generalize the definition of Bernstein polynomial to any
polarized toric \kahler variety, and generalize the calculations
of the previous section.

We recall that a toric \kahler manifold is a \kahler  manifold
$(M, J, \omega)$ on which the  complex torus $(\C^*)^m$ acts
holomorphically with an open orbit $M^o$.  We  assume that $M$ is
projective and that  $P$  is a Delzant polytope,  i.e. a convex
integral polytope in $\R^m$ with the property that each vertex is
contained in exactly $m$ facets, and the normals to the $m$ facets
at each vertex form a $\Z$-basis for a lattice $\Gamma \subset
\R^m$ so that $\T = \R^m/\Gamma$ is the torus acting on $M_P$. The
convex polytope  $P$ is defined by a set of inequalities of
\begin{equation} \label{ELLDEF} \langle x, v_r\rangle\geq \lambda_r, ~~~r=1, ..., d, \end{equation}  where
$v_r$ is a primitive element of the lattice and inward-pointing
normal to the $r$-th $(n-1)$-dimensional face of $P$.

We denote by $\T = (S^1)^m$ the real torus underlying $(\C^*)^m$.
By a toric \kahler metric we mean a \kahler metric $\omega$
invariant under $\T$. We assume that $\frac{1}{\pi}\om$ is a de
Rham representative of the Chern class $c_1(L)\in H^2(M,\R)$ of an
invariant holomorphic line bundle  $L \to M$.  We  let $h$ denote
the Hermitian metric on $L$ inducing the Chern connection with
curvature $(1,1)$ form $\omega_h = \omega$. Here, given a
Hermitian metric $h$,
\begin{equation}\label{curvature}\omega_h= - \frac{\sqrt{-1}}{2} \ddbar
\log\|e_L\|_h^2\;,\end{equation} where $e_L$ denotes a local
holomorphic frame (i.e. a  nonvanishing section) of $L$ over an
open set $U\subset M$, and $\|e_L\|_h=h(e_L,e_L)^{1/2}$ denotes
the $h$-norm of $e_L$. We often write $\omega $ for $\omega_h$
when the metric is fixed.

 Now fix a  basepoint $m_0$ on the open orbit  and
identify  $M^o \equiv (\C^{*})^{m}$, endowing $M^o$ with the
logarithmic coordinates
$$
z=e^{\rho/2 +i\varphi} \in (\C^{*})^{m},\quad \rho, \varphi \in
\R^{m}.$$  Over the open orbit, $\omega$ has a \kahler potential,
i.e. $\omega = - 2i \ddbar \phi(z)$. The associated Hermitian
metric then has the form $h = e^{- \phi}$. Invariance under the
real torus action implies that $\phi$ only depends on the
$\rho$-variables, hence,
$$\omega = \frac{i}{2} \sum_{j, k} \frac{\partial^2 \phi}{\partial
\rho_k \rho_j} dz_j \wedge d\bar{z}_k.$$  We sometimes subscript
$\omega$ to indicate the associated  hermitian metric or \kahler
potential, e.g.  $\omega= \omega_h= \omega_{\phi}$. By a slight
abuse of notation, we denote the \kahler potential in the
logarithmic coordinates by $\phi(\rho)$. Positivity of $\omega$
implies that $\phi$ is strictly convex of $\rho \in \R^n$.

The real torus $\T$ acts on $(M, \omega)$ in  a Hamiltonian
fashion with respect to $\omega$, and its moment map $\mu_{\phi} =
\mu_h$  with respect to $\omega_{\phi} = \omega_h$ is defined by
\begin{equation} \label{MMDEF} \mu_{h} (z_1, \dots, z_m) =
\nabla_{\rho} \phi(\rho_1, \dots, \rho_m), \;\;\; (z = e^{\rho/2 +
i \theta}).
\end{equation} The {\it symplectic potential} $u_{\phi}$
associated to the \kahler potential
 is defined to be the
Legendre-dual of $\phi$, defined as follows: for $x \in P$ there
is a unique $\rho$ such that $\mu_{\phi}(e^{\rho/2}) =
\nabla_{\rho} \phi = x$. Then the Legendre transform is defined to
be the convex function
\begin{equation} \label{SYMPOTDEF} u_{\phi}(x) = \langle x, \rho \rangle -
\phi(\rho), \;\;\; e^{\rho/2} = \mu_{\phi}^{-1}(x)
\end{equation}  on $P$.

There exists a `canonical' \kahler metric and symplectic
potential, defined as follows: Let $l_r: \mathbf{R}^n\rightarrow
\mathbf{R}$ be the affine functions,
$$\ell_r(x)=\langle x, v_r\rangle-\lambda_r.$$
Then the  canonical symplectic potential is defined by
\begin{equation} \label{CANSYMPOT} u_0(x) = \sum_k \ell_k(x) \log
\ell_k(x),
\end{equation}
which in turn  corresponds to a canoncial \kahler potential
\cite{G, A}.  Every symplectic potential  has the same
singularities on the boundary $\partial P$ as the  canonical
symplectic potential.

We denote by $G_{\phi} = \nabla^2_x u_{\phi}$ the Hessian of the
symplectic potential. It has simple poles on $\partial P$.  We
also denote by $H_{\phi}(\rho) = \nabla^2_{\rho} \phi(e^{\rho/2})$
the Hessian of the \kahler potential on the open orbit in $\rho$
coordinates. By Legendre duality, \begin{equation} \label{GINV}
H_{\phi}(\rho) = G_{\phi}^{-1}(x), \;\; \mu_{\phi} (e^{\rho/2}) =
x.
\end{equation}

We now let $(L, h) \to M$ denote the invariant Hermitian  line
bundle with curvature $\omega_h = \omega$.
 A natural basis of the
space of holomorphic sections $H^0(M, L^N)$ associated to the
$N$th power of $L \to M$ corresponds to monomials $z^{\alpha}$
where $\alpha$ is a lattice point in the $N$th dilate of the
polytope, $\alpha \in NP \cap \Z^m.$ The hermitian metric $h$ on
$L$ induces inner products $Hilb_N(h)$ on $H^0(M, L^N)$, defined
by
$$\langle s_1, s_2 \rangle_{h^N} = \int_M (s_1(z), s_2(z))_{h^N}
\frac{\omega_h^m}{m!}. $$ The monomials are orthogonal with
respect to any such toric inner product and have the norm-squares
\begin{equation} \label{QFORM} Q_{h^N}(\alpha) = \int_{\C^m} |z^{\alpha}|^2 e^{-
N \phi(z)} dV_{\phi}(z), \end{equation} where $dV_{\phi} = (i
\ddbar \phi)^m/ m!$. In terms of the symplectic potential,
\begin{equation} \label{SPNORM} Q_{h^N}(\alpha) = \int_P e^{ N
(u_{\phi}(x) + \langle \frac{\alpha}{N} - x, \nabla u_{\phi}(x)
\rangle} dx. \end{equation}

The Bergman-\szego kernels for this hermitian metric are the
orthogonal projections with respect to $Hilb_N(h)$ to $H^0(M,
L^N)$. If we denote the sections corresponding to the monomials by
$S_{\alpha}$ then,
$$\Pi_{h^N}(z,w) = \sum_{\alpha \in N P} \frac{S_{\alpha}(z)
\otimes S_{\alpha}(w)^*}{Q_{h^N}(\alpha)}. $$

The following definition generalizes the formula of (\ref{FS}) to
any toric \kahler manifold.

\begin{defin} \label{BB} Let $f \in C(\bar{P})$. The $N$th  normalized (Bergman-)Bernstein
polynomial approximation  to $f$ with respect to the hermitian
metric $h$ on $L \to M$ is defined by
$$\begin{array}{lll}  B_{h^N} f(x) && =
  \frac{1}{\Pi_N(z,z)} \ncal_{h^N} f(x),\;\; \mbox{where} \\ && \\
   \ncal_{h^N} f(x)  &&=  \sum_{\alpha \in NP}
 f(\frac{\alpha}{N}) \frac{ e^{N
  \left(u_{\phi} (x) + \langle
\frac{\alpha}{N} - x, \nabla u_{\phi}(x) \rangle \right)}
}{Q_{h^N}(\alpha)}
 . \end{array}$$
\end{defin}

 As in the
classical case, Bernstein polynomials are closely related to
certain probability measures on $\bar{P}$. We  define
\begin{equation}
\label{MUN} \mu_N^z: = \sum_{\alpha \in NP} \frac{
\pcal_{h^N}(\alpha, z)}{\Pi_N(z,z)} \delta_{\frac{\alpha}{N}},
\end{equation}
 where $\pcal_{h^N}(\alpha, z)$ denote the
 Fourier coefficients of the Bergman kernel with respect to the
 $\T$,
\begin{equation} \label{PHK} \pcal_{h^N}(\alpha, z): =
\frac{|z^{\alpha}|^2 e^{- N \phi(z)}}{Q_{h^N}(\alpha)}.
\end{equation}

\begin{prop} \label{ID} Let $f \in C(\bar{P})$ and let $x = \mu_{\phi}(z)$ and
let $h = e^{- \phi}$.  Then,
 $$ \begin{array}{lll} B_{h^N}
  f(x) &= & \int_P f(y) d\mu^z_N(y) \\ && \\
  & = &   \sum_{\alpha \in NP}
 f(\frac{\alpha}{N}) \frac{ \pcal_{h^N}(\alpha, z)}{\Pi_N(z,z)},  \\ && \\
  & = &  \frac{1}{\Pi_N(z,z)} \sum_{\alpha \in NP}
 f(\frac{\alpha}{N}) \frac{ e^{N (u_{\phi} (x) +
\langle \frac{\alpha}{N} - x, \log \mu_{\phi}^{-1}(x )\rangle}
}{Q_{h^N}(\alpha)}
 .\end{array}$$
\end{prop}

\begin{proof} The first two equalities are obvious from the
definition. The third equality generalizes the identity
(\ref{AGREE}):
\begin{equation} \label{AGREEa} |z^{\alpha}|^2 e^{- N \phi(z)} = e^{N
  (u_{\phi} (x) + \langle
\frac{\alpha}{N} - x, \log \mu_{\phi}^{-1}(x) \rangle)}, \;\;
\mbox{when}\;\; \mu_{\phi}(z) = x. \end{equation} As in the case
of the Fubini-Study metric, the identity splits into two
identities on the open orbit,
\begin{equation} \label{TWOIDS} |z^{\alpha}|^2 = e^{\langle
\alpha, \rho \rangle}, \;\;\; e^{- N \phi(z)} = e^{N
  (u_{\phi} (x) - \langle x, \log \mu_{\phi}^{-1}(x) \rangle)}.
  \end{equation}
The first follows from the fact that
 \begin{equation} \nabla_x u_{\phi}(x) = \log \mu_{\phi}^{-1} (x) =
 \rho,
 \end{equation}
since  by (\ref{SYMPOTDEF}),
 $
 \nabla_x u_{\phi}(x) = \rho + \langle x, \nabla_x \rho \rangle -
 \langle \nabla \phi(\rho), \nabla_x \rho \rangle = \rho, $
as $\nabla \phi(\rho) = x. $ The second then follows from the fact
that $\phi(\rho)$ and $u_{\phi}(x)$ are Legendre duals. The
identity of the Proposition then extends by continuity to the
closure.

\end{proof}

As a simple corollary, we obtain one of the standard properties of
Bernstein polynomials. 

\begin{cor} Let $f \in C(\bar{P})$. Then $\min_{\bar{P}} f \leq
B_N(f)(x) \leq \max_{\bar{P}} f$. \end{cor}

Let us calculate explicitly the numerator polynomials for the
canonical symplectic potential (\ref{CANSYMPOT}) or \kahler form.
 We have,
$$\nabla u_0(x) = \sum_k (\log \ell_k) v_k + \bar{v}, \;\; \bar{v}
= \sum_k v_k. $$ Hence,
$$\langle \frac{\alpha}{N} - x, \nabla u_0 (x)  \rangle = \sum_k
\langle \frac{\alpha}{N} - x, v_k \rangle \log \ell_k + \langle
\frac{\alpha}{N} - x, \bar{v} \rangle, $$ and
$$\begin{array}{lll} e^{N
  \left(u_0(x) + \langle
\frac{\alpha}{N} - x, \nabla u_0 (x) \rangle \right)} & = &
e^{\langle \alpha - N x, \bar{v} \rangle}  \Pi_k (\ell_k(x))^{N
\ell_k(x) + \langle \alpha - N x, v_k \rangle} \\ && \\
& = & e^{\langle \alpha - N x, \bar{v} \rangle}  \Pi_k
(\ell_k(x))^{- N \lambda_k + \langle \alpha, v_k \rangle}
\end{array}
$$
where in the last line we use that $ \ell_k(x) - \langle x,
v_k\rangle = - \lambda_k$. Hence, the numerator of the canonical
Bernstein polynomial may be rewritten as
\begin{equation} \label{NUMERATOR} \begin{array}{lll}  \ncal_{h^N} f(x) && =
\sum_{\alpha \in NP}
 f(\frac{\alpha}{N}) \frac{1}{Q_{h_{can}^N}(\alpha)} e^{\langle \alpha - N x, \bar{v} \rangle}  \Pi_k
(\ell_k(x))^{- N \lambda_k + \langle \alpha, v_k \rangle},
\end{array} \end{equation}
which closely resembles the classical cases (where also $\bar{v} =
0$). Here, $$Q_{h_{can}^N}(\alpha) = \int_P e^{N (u_0(x) + \langle
\frac{\alpha}{k} - x, \nabla u_0(x) \rangle} dx
$$ is the norming constant with respect to the canonical
symplectic potential.

In general, the symplectic potential has the form
\begin{equation} \label{CANSYMPOTa} u_{\phi}(x) = u_0(x) + g_{\phi}(x) =  \sum_k \ell_k(x) \log
\ell_k(x) + g_{\phi}(x),
\end{equation}
where  $g_{\phi} \in C^{\infty}(\bar{P})$ is smooth up the
boundary \cite{G,A,D2}. Hence the $\alpha$ term gets multiplied by
the additional factor
$$ e^{N
  \left(g_{\phi} (x) + \langle
\frac{\alpha}{k} - x, \nabla g_{\phi} (x) \rangle \right)}. $$

In Definition \ref{BBa}, Bernstein polynomials were normalized by
 dividing by $\Pi_{h^N}(z,z)$. It follows that
 $B_{h^N}(1) \equiv 1$ as for classical Bernstein polynomials.  However, in
  the classical cases, $\Pi_{h^N}(z,z)$ is constant so the normalization
  of the polynomial is still a polynomial.  In general, however, $\Pi_{h^N}(z,z)$
  is not constant on the diagonal and therefore the quotient is
  rarely a polynomial in the usual sense.  In special cases, Bernstein polynomials are
polynomials in the variables $x = \mu(z)$, but  this depends on
the properties of the moment map of the \kahler metric.

 One might prefer to
 normalize by dividing  the numerator polynomial by the  dimension polynomial $d_N = \dim
 H^0(M, L^N)$,
 \begin{equation} \label{NORMAL2} \begin{array}{l}  \hat{B}_{h^N} f(x) = \frac{1}{d_N} \ncal_{h^N} f(x).
  \end{array} \end{equation} For the canonical metric, one would
 then
 have the {\it canonical Bernstein polynomials}
 \begin{equation} \label{NUMERATOR} \begin{array}{lll}   \hat{B}_{h^N} f(x) && =
\frac{1}{d_N} \sum_{\alpha \in NP}
 f(\frac{\alpha}{N}) \frac{1}{Q_{h^N}(\alpha)} e^{\langle \alpha - N x, \bar{v} \rangle}  \Pi_k
(\ell_k(x))^{- N \lambda_k + \langle \alpha, v_k \rangle},
\end{array} \end{equation}
which are visibly polynomials when $\bar{v} = 0$.
 However, this is essentially an aesthetic decision based on
how seriously one wants to take the term `polynomial'.  Either
definition has the same value in terms of making approximations.
Our view is that the term `polynomial' should have the same
general sense in the  \kahler context as `algebro-geometric'
approximations do  in the Yau-Tian-Donaldson program.

To our knowledge the only previously studied cases are the
Bernstein polynomials for the simplex (\ref{NDBERNDEF}) or cube,
$$B_N(f)(x) = \sum_{0 \leq i_1, \dots, i_d \leq N}
f(\frac{i_1}{N}, \dots, \frac{i_d}{N}) \Pi_{k = 1}^d {N \choose
i_k} x_k^{i_k} (1 - x_k)^{N - i_k}. $$ Here, $(x_1, \dots, x_d)
\in [0, 1]^d$. The classical Bernstein polynomials $B_N(f)$ are
distinguished among other polynomial approximations by
simultaneously approximating the derivatives and also by
preserving certain shape and convexity properties (at least, in
dimension one). It might be interesting to explore the shape
preserving properties of general Bernstein polynomials. We also
note that Bernstein polynomials admit holomorphic extensions to
complex neighborhoods of the polytope $P$ when the symplectic
potential function $g_{\phi}$ is real analytic.

\section{\label{BPTO} Bernstein polynomials, Toeplitz operators and Berezin
symbols}

In this section, we prove formula  (\ref{BN}) and also establish
some basic properties of Bernstein polynomials.

 The proof of (\ref{BN})  is simply a matter of
unwinding the definitions. The Bergman kernel is a section of the
bundle $(L^N) \otimes (L^N)^* \to M \times M$. It is simpler to
deal with scalar kernels, and so we lift the Bergman kernel to a
kernel $\hat{\Pi}_N(x,y)$ on the unit circle bundle $X \to M$ with
respect to $h$ in the dual line bundle $L^*$. In other words, $X =
\partial D^*_h $ is the boundary of the unit disc bundle with respect to $h$  in the dual line
bundle $L^*$. We use local product coordinates $x = (z, \theta)
\in M \times S^1$ on $X$
 where $x = e^{i \theta}
\frac{e(z)}{||e(z)||}$ in terms of a local holomorphic frame
$e(z)$ for $L$.  When working on $M$ we tacitly use the
representative of $\Pi_{h^N}$ relative to the frame $e(z)^N$ of
$L^N$. For the sake of brevity, we will not review the definitions
but refer to \cite{STZ} for the relevant background.

The space $H^0(M, L^N)$ is naturally isomorphic to the space
$H^2_N(X)$ of CR holomorphic functions transforming by $e^{i N
\theta}$ under the $S^1$ action of the circle bundle $X \to M$. We
denote by $s \to \hat{s}$ the lift of a section to an equivariant
CR function and by  $\hat{\Pi}_{h^N}(x,y)$ the lifted \szego
kernel, i.e. the orthogonal projection from $L^2(X) \to H^2_N(X)$.
The monomial sections $s_{\alpha}$ which equal $z^{\alpha}$ on the
open orbit lift to equivariant functions $\hat{s}_{\alpha}$ on
$X$.

By the standard linearization of geometric quantization (reviewed
in this context in \cite{STZ}), the  $\T$ action lifts to $X$ as
contact transformations of the Chern connection form associated to
$h$. For the sake of completeness, let us recall the lift of the
torus action to $H^2_N(X)$, and its linearization on $H^0(M,
L^N)$: The generators $\frac{\partial}{\partial \theta_j}$ of the
$\T$ action on $M$ lift to contact vector fields
$\Xi_1,\dots\Xi_m$ on $X$. There is a natural contact 1-form
$\alpha$ on $X$ defined by the  Hermitian connection 1-form, which
satisfies $d \alpha = \pi^* \omega$.  The horizontal lifts of the
Hamilton vector fields $\xi_j$ are then defined by
$$\pi_* \xi^h_{j} = \xi_j,\;\;\; \alpha(\xi^h_j) = 0,$$
and the contact vector fields $\Xi_j$ are given by:
\begin{equation*} \Xi_j = \xi^h_j + 2 \pi i \langle
\mu \circ\pi, \xi_j^* \rangle \frac{\partial}{\partial \theta}
=\xi^h_j + 2\pi i (\mu \circ\pi)_j\, \frac{\partial}{\partial
\theta},
\end{equation*}
where $\mu$ is the moment map corresponding to $h$,  and where
$\xi_j^* \in \R^m$ is the element of the Lie algebra of $\T$ which
acts as $\xi_j$ on $M$.

 It follows that the vector fields act as differential operators
 on the CR Hardy spaces,  $ \Xi_j:H^2_N(X) \to H^2_N(X)$ satisfying
\begin{equation}\label{Xi}(\Xi_j \hat S)
(\zeta)= \frac{\d}{\d\phi_j}  \hat
S(e^{i\phi}\cdot\zeta)|_{\phi=0}\;,\quad \hat S\in
\ccal^\infty(X_P^c)\;.
\end{equation}
Furthermore, the generator of the $S^1$ action acts on these
spaces and
\begin{equation}\label{dtheta} \frac{\d}{\d\theta} :\hcal^2_N(X_P^c) \to
\hcal^2_N(X_P^c)\;,\qquad \frac{1}{i}\frac{\d}{\d\theta}\hat s_N =
N\hat s_N \quad \mbox{for }\ \hat s_N\in
\hcal^2_N(X_P^c)\;.\end{equation}

Since by (\ref{Xi}), the operators $\Xi_j$ act by translating
functions by the $\T$ action lifted to $X$, we henceforth denote
$\frac{1}{i} \Xi_j$   by $D_{\theta_j}$. Then for    $1\le j\le
m$, the lifted monomials $\hat\chi_{\alpha} \in H^2_N(X)$ are
joint eigenfunctions of these commuting operators,
$$D_{\theta_j} \hat\chi_{\alpha} = \alpha_j \hat\chi_{\alpha}, \;\; \forall \alpha \in NP.$$
The dilation $P \to NP$ is best viewed in terms of constructing a
conic set of eigenvalues in one higher dimension by adding the
operator
\begin{equation}
\hat{I}_{m+1}=\frac{p}{i}\frac{\partial}{\partial \theta}
-\sum_{j=1}^{m} D_{\theta_j}. \label{Qtorus}
\end{equation}
The monomials $\wh{\chi}_{\wh{\alpha}}$ are then the  joint
eigenfunctions of these $(m + 1)$ commuting operators  and  we
define the `homogenization' $\wh{NP} \subset \mathbb{Z}^{m+1}$ of
the lattice points in the polytope $NP$ to be the set of all
lattice point $\wh{\alpha}^{N}$ of the form
\begin{equation}
\wh{\alpha}^{N}=\wh{\alpha}:=(\alpha_{1},\ldots,\alpha_{m},N
-|\alpha|),\quad \alpha=(\alpha_{1},\ldots,\alpha_{m}) \in NP \cap
\mathbb{Z}^{m},
\end{equation}

 Given
$f \in C^{\infty}(\R^m)$, we now  define $f(D_{\theta})$ on
$L^2(X)$ by the spectral theorem for $m$ commuting operators, i.e.
$$f(D_{\theta}) = \int_{\R^m} \hat{f}(\xi) e^{i \langle \xi,
D_{\theta} \rangle } d\xi, \;\; \mbox{ where}\;\; \langle \xi,
D_{\theta} \rangle = \sum_j \xi_j D_{\theta_j}. $$

We then have
\begin{equation} f(N^{-1} D_{\theta}) \hat{s}_{\alpha}  =
f(\frac{\alpha}{N}) \hat{s}_{\alpha}
\end{equation}
Since $\hat{\Pi}_{h^N}(\hat{z}, \hat{w}) = \sum_{\alpha \in N P}
\hat{s}_{\alpha}(\hat{z}) \overline{\hat{s}_{\alpha}(\hat{w})}, $
we have
\begin{equation}  f(N^{-1}
D_{\theta}) \hat{\Pi}_{h^N}(e^{i \theta} \hat{z},  \hat{w})  =
\sum_{\alpha \in NP} f(\frac{\alpha}{N}) \hat{s}_{\alpha}(\hat{z})
\overline{\hat{s}_{\alpha}(\hat{w})}. \end{equation} It follows
that
\begin{equation}  f(N^{-1}
D_{\theta}) \hat{\Pi}_{h^N}(e^{i \theta} \hat{z},  \hat{w})
|_{\hat{z} = \hat{w}} = \sum_{\alpha \in NP} f(\frac{\alpha}{N})
\left|\hat{s}_{\alpha}(\hat{z})\right|^2. \end{equation} The right
hand side is constant along the orbits of the $S^1$ action and may
be identified with a function of $z \in M$. On $M$ we have
$\left|\hat{s}_{\alpha}(\hat{z})\right|^2 =
||s_{\alpha}(z)||_{h^N}^2$ and by Proposition \ref{ID} we obtain
the definition of the numerator  polynomials when we substitute $z
= \mu^{-1}_h(x)$. Equivalently,
\begin{equation} \label{BEREZIN} \ncal_{h^N}(f)(x) =
\; \left(\hat{\Pi}_{h^N} f(N^{-1} D_{\theta}) \hat{\Pi}_{h^N}
\right) (e^{i \theta} z,z)|_{\theta = 0; z = \mu_h^{-1}(x)},
\end{equation} where the right side is the Berezin symbol of the
Toeplitz operator $\hat{\Pi_{h^N}} f(N^{-1} D_{\theta})
\hat{\Pi}_{h^N}$. We then divide by $\Pi_{h^N} (z,z)$ to obtain
the Bernstein polynomials.

\section{Proof of Theorems \ref{BBa} and  \ref{NUM}}

We now use the Boutet de Monvel - Sj\"ostrand parametrix
\cite{BSj, BerSj, BBSj} to obtain a complete asymptotic expansion
for the Bernstein polynomials from (\ref{BEREZIN}).  There now
exist many expositions of the construction and properties of this
parametrix, so we will only briefly recall the essential elements
in the case of toric varieties \cite{SoZ,STZ}.  We also use the
notation $x, y$ for points of $X$, hoping that no confusion with
coordinates on $P$ will occur.

We first recall that, on the diagonal,  the Bergman-\szego kernel
has a complete asymptotic expansion,
\begin{equation} \label{TYZ}  \Pi_{h^N}(z,z) = \sum_{i=0}^{d_N} ||S^N_i(z)||_{h_N}^2 =  \frac{N^m}{\pi^m}\left[1+a_1(z) N^{-1} +a_2(z)
N^{-2}+\cdots \right]\,,
\end{equation}
for certain smooth coefficients $a_j(z)$. In fact,
\begin{equation}  \left\{
\begin{array}{l}
a_1=\frac 12 S \\
a_2=\frac 13\Delta S +\frac{1}{24}(|R|^2-4|Ric|^2+3 S^2)\\
\end{array}
\right.
\end{equation}
where $R, Ric$ and $S$ denotes the curvature tensor, the Ricci
curvature and the scalar curvature of $\om_h$, respectively, and
$\Delta$ denotes the Laplace operator of $(M,\om_h)$.; see
\cite{Z2,Lu,BSj,BBSj}.

Off the diagonal we have the following expansion:

\begin{prop} \label
{PIKZW} For any $C^{\infty}$  positive hermitian line bundle $(L,
h)$, there exists a semi-classical amplitude in the parameter
$N^{-1}$, $s_N(z, w) \sim N^m s_0(z,w) + N^{m-1} s_1(z, w) +
\cdots $, such that
$$\Pi_{h^N} (z, w) = e^{ N (\phi(z,w) - \frac{1}{2}(\phi(z) +
\phi(w)))} s_N(z,w) + O(N^{- \infty}), $$ where $\phi$ is a smooth
local \kahler potential for $h$, and where $\phi(z,w)$ is the
almost-analytic extension of $\phi(z) = \phi(z, \bar{z})$.
\end{prop}

Since the  local \kahler potentials (e.g. the \kahler potential on
the open orbit)  are invariant under the $\T$ action, they can be
expressed in the form $F(|z|^2)$ where $F \in C^{\infty}(\R)$.
  We
denote by  $F(z \cdot \bar{w})$  the almost analytic extension of
$F$. Thus, we have:
\begin{prop}\label{SZKTV}  For any hermitian  toric positive line bundle over a toric
variety, the \szego kernel for the metrics $h_{\phi}^N$ have the
asymptotic expansions in a local frame on $M$,
$$\Pi_{h^N}(z, w) \sim e^{N \left(F(z \cdot \bar{w}) - \frac{1}{2} (F(||z||^2)
+ F(||w||^2)) \right) } A_N(z,w) \;\; \mbox{mod} \; N^{- \infty},
$$ where $A_N(z,w) \sim N^m \left(1 + \frac{a_1(z,w)}{N} + \cdots\right) $ is a semi-classical symbol of order $m$. \end{prop}

We now prove  Theorems \ref{BBa} and  \ref{NUM}.

\begin{proof}

We then apply the geometric quantizations of the torus action to
get, by Definition \ref{BB},
$$e^{i \langle \xi, N^{-1} D_{\theta} \rangle} \Pi_{h^N}(e^{i \theta} z,  w)|_{z = w; \theta = 0}=
\sum_{\alpha \in N P \cap \Z^m} \frac{e^{i \langle N^{-1} \alpha,
\xi  \rangle} |z^{\alpha}|^2 e^{- N F(
|z|^2)}}{\QQ_{h^N}(\alpha)}.$$ By (\ref{BEREZIN}), we obtain
$\ncal_{h^N}f(x)$ by integrating the  right side against
$\hat{f}(\xi)$.  We note that in general $e^{i \langle \xi, N^{-1}
D_{\theta} \rangle} \psi( e^{i \theta} w) |_{\theta = 0} = \psi(
e^{i (\theta + \frac{\xi}{N})} w) |_{\theta = 0} =  \psi( e^{i
\frac{\xi}{N}} w)$ Performing the same transformation on the
parametrix gives,
\begin{equation} \label{BKFORM} \ncal_{h^N} (f)(x) \sim  \int_{\R^m} \hat{f}(\xi)  e^{
N (F(e^{iN^{-1} \xi} |z|^2) - F(|z|^2))} A_N \big( z,
 e^{i \frac{\xi}{N}} z\big) d
\xi,\end{equation} where $\sim $ means that the difference is a
function which decays rapidly in $N$ along with its derivatives.
Such a remainder may be neglected if we only consider expansions
modulo rapidly decaying functions of $N$.

We have,
\begin{equation} \label{PHASEANALYSIS} \begin{array}{lll}
F_{\C}( e^{i N^{-1} \xi} |z|^2) - F( |z|^2) &=& \int_0^1
\frac{d}{dt} F_{\C}( e^{i t N^{-1}\xi} |z|^2) dt \\ &&\\ & = & i
N^{-1} \int_0^1 \langle \nabla_{\xi} F(e^{i t N^{-1}\xi + \rho}),
\theta \rangle
dt \\ && \\
& = & i N^{-1} \langle \nabla_{\xi} F(e^{\rho}),  (i \xi) \rangle
 + (iN)^{-2} \int_0^1
(t - 1) \nabla_{\rho}^2 (F(e^{i t N^{-1} \xi +  \rho})) (i
\xi)^2/2
dt  \\ && \\
& = & i N^{-1} \langle \mu(z),  \xi \rangle + (iN)^{-2}
\nabla_{\rho}^2
(F(e^{\rho})) (i \xi)^2  + R_3(\xi, N, \alpha) \\ && \\
& = & i N^{-1}  \langle \mu(z),  \xi \rangle + (iN)^{-2} \langle
H_{z} \xi, \xi \rangle + N^{-2} R_3(\xi, N, z),
\end{array} \end{equation}
  where
\begin{equation} \label{REMAINDERa} R_3(\xi, N, z) : = N^{-3} \int_0^1 (t - 1)^2
\nabla_{\rho}^3 (F(e^{i t \xi +  \rho})) (i \xi)^3/3!,
\end{equation} and where $H_{z} = \nabla^2 F(|z|^2) = \nabla^2 \phi(e^{\rho})$ is the Hessian in the notation
(\ref{GINV}).
Hence, (\ref{BKFORM}) takes the form
\begin{equation} \begin{array}{lll} B_N(f)(x) & \sim &  \int_{\R^m} \hat{f}(\xi)  e^{
i  \langle \mu(z),  \xi\rangle}  e^{ (iN)^{-1} \langle H_{z} \xi,
\xi \rangle + N^{-1} R_3(\xi, N, z)} A_N \big( z,
 e^{i \frac{\xi}{N}} z , 0, N \big) d
\theta \end{array} \end{equation} and  by Taylor expanding the
factor $e^{ (iN)^{-1} \langle H_{z} \xi, \xi \rangle + N^{-1}
R_3(\theta, N, z)} $ one obtains an amplitude $\tilde{A}_N$ such
that
\begin{equation} \begin{array}{lll} \ncal_{h^N}(f)(x) & \sim &  \int_{\R^m} \hat{f}(\xi)  e^{
i \langle \mu(z),  \xi \rangle } \tilde{A}_N \big( z,
 e^{i \frac{\xi}{N}} z , 0, N \big) d
\theta.
\end{array}
\end{equation}

The amplitude $\tilde{A}_N$   has an expansion of the form,
$$\tilde{A}_N\big( z,
 e^{i \frac{\xi}{N}} z , 0, N \big) = N^m a_0  + N^{m-1} a_1 + O(N^{m-1}),$$
 for various smooth coefficients $a_j(z)$; the first one is
 constant. If we divide by $\Pi_{h^N}(z,z)$ we cancel the constant
 and by expanding the denominator we obtain,
\begin{equation} \label{BNEXP} \begin{array}{lll} \ncal_{h^N} (f)(x) & \sim &  N^m  f(\mu(z)) + N^{m-1} \left( i^{-1} \langle H_{z} D_x, D_x
\rangle f (\mu(z)) + a_1(z, z) f(\mu(z)) \right) + O(N^{m-2}),
\end{array}
\end{equation}
 Since $\mu(z)
= x$ we obtain Theorem  \ref{NUM}. Dividing by $\Pi_{h^N}(z,z)$
and using (\ref{TYZ}) completes the proof of  Theorem \ref{BBa}.
\end{proof}

It is difficult (but possible) to calculate the coefficients in
explicit geometric terms by this method. In the next section, we
will reduce the calculation to the known calculation of Bergman
kernel expansion coefficients.

\subsection{\label{CORPROOF} Proof of Corollary \ref{EM}}
To prove the Corollary, we integrate the expansion (\ref{BNEXP})
over $P$ to obtain
\begin{equation} \begin{array}{lll} \int_P \ncal_{h^N} (f)(x) dx
& = &  N^m \int_P f(x) dx   + O(N^{m-1}),
\end{array}
\end{equation}
where the lower order terms could be computed from the expansion.
But we postpone their evaluation until the next section.

\section{\label{BKBE} Bergman kernel expansion and geometric expressions for the Bernstein expansion of Theorem \ref{BB}}

In this section, we give a second proof of the convergence of
Bernstein polynomials which is based on  one of their essential
features: the localization of the sum over $\frac{\alpha}{N} \in P
\cap \frac{1}{N} \Z^m$ around the image of $z$ under the moment
map. This is well-known and various expositions can be found in
 \cite{Ho,K,L}; see also \cite{D} Lemma 6.3.5). This approach
 reduces the calcluate  the lower order terms in the Bernstein polynomial
expansion in terms of the Bergman kernel expansion in \cite{Z,Lu}
and elsewhere.

 The
relevant Localization Lemma was proved in \cite{SoZ}. We use a
notation similar to \cite{Ho}.

\begin{lem} (Localization of Sums)  \label{LOCALIZATION}  \cite{SoZ} Let
$f \in C(\bar{P})$.  Then, there exists $C > 0$ so that
$$ \sum_{\alpha \in N P \cap \Z^m
} f(\frac{\alpha}{N})
\frac{|S_{\alpha}(z)|^2_{h^N}}{\QQ_{h^N}(\alpha)} = \sum_{\alpha:
|\frac{\alpha}{N} - \mu_h(z)| \leq  N^{-1 + \delta} }
f(\frac{\alpha}{N})
\frac{|S_{\alpha}(z)|^2_{h^N}}{\QQ_{h^N}(\alpha)} \; + \;
O_{\delta} (N^{- C}). $$
\end{lem}

Hence, it is natural to Taylor expand $f$ around $\mu_h(z)$ to
obtain
$$f(\alpha/N) = \sum_{\nu < 2 M} f^{(\nu)}(\mu_h(e^{\rho}))
(\frac{\alpha}{N} - \mu_h(e^{\rho}))^{\nu} / \nu! + R_M(f,
e^{\rho}, \frac{\alpha}{N}),
$$
where $R_M$ is the $M$th order Taylor remainder. We then have,

\begin{equation} \label{EXP2}  \begin{array}{lll} \ncal_{h^N} f(x)  & = &\sum_{\beta: |\beta| \leq
M} \frac{1}{\beta!} D^{\beta}_x f(\mu(z)) \left( \sum_{\alpha\in N
\PP\cap \mathbf{Z}} (\frac{\alpha}{N} - \mu_h(z))^{\beta}
\frac{|S_{\alpha}|^2_{h^N}}{\QQ^N(\alpha)} \right) \\ && \\ && +
\rcal (M, N, z),\end{array} \end{equation} where the remainder
 is obtained by summing $R_M(f, e^{\rho},
\frac{\alpha}{N})$ in the variable $\frac{\alpha}{N}$.

 To prove the main result, we need to
study the special functions
\begin{equation} \label{INU} I_{h^N}^{\nu} (z): =
\sum_{\alpha\in N \PP\cap \mathbf{Z}} (\frac{\alpha}{N} -
\mu_h(z))^{\nu} \frac{|S_{\alpha}(z)|^2_{h^N}}{\QQ_h^N(\alpha)} =
\sum_{\alpha\in \PP\cap \mathbf{Z}} (\frac{\alpha}{N} - \mu_h
(e^{\rho/2}))^{\nu} \frac{e^{\langle \alpha, \rho \rangle - N
\phi_t(e^{\rho/2})}}{\QQ_h^N(\alpha)}.
\end{equation}

\begin{prop} Uniformly for $z \in M$ we have: \begin{equation} I_{h^N}^{\nu}(z)
 = O(N^{m- \nu/2} (\log N)^{\nu}).
\end{equation}\end{prop}

\begin{proof}

  The Localization lemma implies that
$$I_{h^N}^{\nu}(z) = \sum_{\alpha \in N P  \cap \Z^m : |\frac{\alpha}{N} - \mu_h(z)| \leq \frac{C \log
N}{N}}  (\frac{\alpha}{N} - \mu_h(z))^{\nu}
\frac{|S_{\alpha}(z)|^2_{h^N}}{Q^N(\alpha)}  + O(N^{-C}). $$ In
the domain of summation we then have,
$$(\frac{\alpha}{N} - \mu_h (e^{\rho/2}))^{\nu} = (\frac{\log
N}{{\sqrt N}})^{\nu},$$ and this implies the statement.

\end{proof}

We can explicitly evaluate these functions by relating them to
derivatives of the Bergman-\szego kernels. The following Lemma was
also used in \cite{SoZ}. We  employ a tensor product notation
$(\frac{\alpha}{N} - \mu_h(e^{\rho/2}))^{\otimes 2}_{ij}$ for
$(\frac{\alpha_i}{N} - \mu_h(e^{\rho/2})_i) (\frac{\alpha_j}{N} -
\mu_h(e^{\rho/2})_j).$ In the following, we implicitly assume that
$z$ lies in the open orbit and express it as $z = e^{\rho/2 + i
\theta}$. Similar formula hold at the boundary as well where the
vector fields $\frac{\partial}{\partial \rho_j}$ are replaced by
derivatives in affine coordinates.  For the sake of brevity we
refer to \cite{SoZ} for the modifications to the formulae around
the boundary.

\begin{prop} \label{COMPAREPIT} We have:

\begin{enumerate}

\item $ \sum_{\alpha \in N P \cap \Z^m} (\frac{\alpha}{N} -
\mu(e^{\rho/2})) \frac{e^{\langle \alpha, \rho \rangle  - N
\phi(e^{\rho/2})} }{\QQ_{h^N}(\alpha)} = \frac{1}{N} \nabla_{\rho}
\Pi_{h^N}(e^{\rho/2}, e^{ \rho/2} ); $

\item $   \sum_{\alpha \in N P \cap \Z^m} (\frac{\alpha}{N} -
\mu(e^{\rho/2}))^{\otimes 2}_{ij} \frac{e^{\langle \alpha, \rho
\rangle - N \phi(e^{\rho/2})} }{\QQ_{h^N}(\alpha)} =  \frac{1}{N}
\Pi_{h^N}(e^{\rho/2}, e^{\rho/2}) \nabla^2_{\rho} \phi    +
\frac{1}{N^2} \nabla^2 \Pi_{h^N}(e^{\rho/2}, e^{\rho/2}). $

\end{enumerate}
\end{prop}

\begin{proof}

 To prove (1), we differentiate (\ref{MMDEF})  to obtain
$$\begin{array}{lll} \nabla_{\rho} \Pi_{h^N}(e^{\rho/2}, e^{ \rho/2} ) & = &
N \sum_{\alpha \in N P \cap \Z^m} (\frac{\alpha}{N} -
\mu(e^{\rho/2})) \frac{e^{\langle \alpha, \rho \rangle  - N
\phi(e^{\rho/2})}  \Pi_{h^N}(e^{\rho/2}, e^{\rho/2})
}{\QQ_{h^N}(\alpha)}.
\end{array}$$

To prove (2), we take a second  derivative of (1)  in $\rho$  to
get
$$\begin{array}{l} \nabla_{\rho}^2  \Pi_{h^N}(e^{\rho/2}, e^{ \rho/2}
)=  - N \nabla \mu_h (e^{\rho/2}))  \Pi_{h^N}(e^{\rho/2}, e^{\rho/2}) \\ \\
+   N^2 \sum_{\alpha \in N P \cap \Z^m} (\frac{\alpha}{N} - \mu_h
(e^{\rho/2}))^{\otimes 2} \frac{e^{\langle \alpha, \rho \rangle -
N \phi(e^{\rho/2})} }{\QQ_{h^N}(\alpha)}.
\end{array}$$

\end{proof}

We now evaluate these functions geometrically:

\begin{prop} \label{COMPAREPIT2} We have:

\begin{enumerate}

\item $I^{(1)}_{h^N}(z) = C_m N^{m-2}  \nabla S (z) + O(N^{m-3});$

\item $I^{(2)}_{h^N}(z)   = N^{m-1}  \nabla^2_{\rho} \phi +
N^{m-2} S(z) \nabla^2_{\rho} \phi   + O(N^{m -3}) $.

\end{enumerate}
\end{prop}

\begin{proof}

From (\ref{TYZ}) it follows that

$$\begin{array}{l} \nabla_{\rho} \Pi_{h^N}(z,z)
=
N^{m-1} \nabla S(z) + O(N^{m-2}), \\  \\
\nabla_{\rho}  \mu_h (z)  \Pi_{h^N}(z,z)
 = N^m \nabla \mu_h + N^{m-1} C_m S(z) \nabla \mu_h +
 O(N^{m-2});\\ \\
 \nabla_{\rho}^2  \Pi_{h^N}(e^{\rho/2}, e^{ \rho/2}) = N^{m-1}
 \nabla^2_{\rho} S(z) + O(N^{m-2}).
\end{array}$$
We also use that  $\nabla \mu_h (e^{\rho/2})) = \nabla^2 \phi$.

\end{proof}

To complete the second proof of Theorem \ref{BBa}, it suffices to
observe that the remainder in (\ref{EXP2}) after expanding to
order $M$ is  $O(N^{m -  M/2} (\log N)^M), $ which follows from
the fact that $\rcal(M, N, z) \leq C_f N^m   I_{h^N}^{\nu +
1}(z).$ Therefore
\begin{equation} \label{LAST} \begin{array}{lll}  \ncal_{h^N}(f)(\mu(z)) & =
& f(\mu(z)) \Pi_{h^N}(z,z)     \\ && \\
& + &   \sum_{|\beta| = 1} D^{\beta} f(\mu(z))
I^{(\beta)}_{h^N}(\mu(z)) \\ && \\
& + & \frac{1}{2} \sum_{|\beta| = 2} D^{\beta} f(\mu(z))
I_{h^N}^{\beta}
(\mu(z)) + O(N^{- 3/2}(\log N)^3)  \\ && \\
& = &  N^m f(\mu(z)) + N^{m -1} \left(f(\mu(z)) S(z) + \nabla
\mu_h \cdot \nabla^2 f(\mu(z)) \cdot  \right) +  O(N^{m -
3/2}(\log N)^3).
\end{array}
\end{equation}

\section{\label{EMa} Dedekind-Riemann sums over lattice points: Proof of Corollary \ref{EM}}

As noted in \S \ref{CORPROOF},  the existence of an asymptotic
expansion for the Riemann sums follows immediately from theorem
\ref{NUM}. However, it is an expansion in terms of integrals of
curvature invariants against derivatives of $f$ over $P$. The
purpose of this section is to prove that the first two terms can
be put in the form stated in Corollary \ref{EM}, and thus to
clarify the relation between the Bernstein and Euler-MacLaurin
approaches to lattice point sums.

  We begin the calculation by
integrating (\ref{LAST}) over $M$ with respect to
$\frac{\omega^m}{m!}$ and recalling that the pushforward to $P$ of
this volume form under the moment map $\mu_h$ is Lebesgue measure
$dx$ on $P$. Also, $\Pi_{h^N}(z,z)$ is constant on  $\T$- orbits,
so $\Pi_{h^N}(\mu^{-1}(x), \mu^{-1}(x))$ is well-defined although
the inverse image is an orbit. The same is true for geometric
functions such as the scalar curvature. We also recall that $d_N =
\dim H^0(M, L^N)$). Then by Proposition \ref{COMPAREPIT2} only the
zeroth and second order terms of the Taylor expansion of $f$
contribute to the $N^{-1}$ term of the Riemann sum expansion, and
we have
\begin{equation} \label{LAST} \begin{array}{lll}  \sum_{\alpha \in N P} f(\frac{\alpha}{N})  & =
&   \int_P f(x) \Pi_{h^N}(\mu^{-1}(x), \mu^{-1}(x))  dx \\ && \\ &
+ & \frac{1}{2} \sum_{|\beta| = 2} \int_P D^{\beta} f(x)
I_{h^N}^{\beta}
(x) dx + O(N^{- 3/2}(\log N)^3)  \\ && \\
& = &  \frac{N^m}{\pi^m} \int_P  f(x) dx +  \frac{N^{m -1}}{\pi^m}
\int_P \frac{1}{2} f(x) S(\mu^{-1}(x)) \\ && \\ && + \frac{1}{2}
\langle \nabla_{\rho} \mu_h (\mu^{-1}(x), \nabla_x^2 f(x) \rangle
dx  + O(N^{m - 3/2}(\log N)^3).
\end{array}
\end{equation}
Here, $\langle \nabla \mu_h,  \nabla^2 f(\mu(z)) \rangle $ denotes
the Hilbert-Schmidt inner product of the tensors.

By  Legendre duality, the Hessians of the \kahler potential and
symplectic potentials are inverses, i.e.
\begin{equation} \nabla_{\rho} \mu_h (\mu^{-1}(x) = (\nabla^2
u_{\phi} (x))^{-1}. \end{equation} Hence,
\begin{equation} \langle \nabla_{\rho} \mu_h (\mu^{-1}(x), \nabla_x^2 f(x) \rangle
dx  = \int_P \sum_{j k} u_{\phi}^{j k} f_{, j k}  dx.
\end{equation}
Further, we recall (cf. \cite{D2,A}) that the scalar curvature of
a toric \kahler metric is given in terms of the symplectic
potential by
\begin{equation} \label{scalarsymp2}
S = -  \sum_{j,k} \frac{\partial^2 u_{\phi} ^{jk}}{\partial x_j
\partial
 x_k}\,,
\end{equation}
where $u_{\phi} ^{jk},\ 1\leq j,k\leq n$ are the entries of the
inverse of the matrix $\nabla^2 u_{\phi}$. See  \cite{D2} (3.1.4).

We now use the following integration by parts formula due to
Donaldson: \begin{lem} \label{DON} (\cite{D2}, Lemma 3.3.5) For
any symplectic potential $u_{\phi}$ and $f \in C^{\infty}$,
$\sum_{j k} u_{\phi}^{j k} f_{, j k} \in L^1(P)$ and
$$\int_P \sum_{j k} u_{\phi}^{j k} f_{, j k} = \int_P \sum_{j k} (u_{\phi}^{j k})_{,jk}
fdx + \int_{\partial P} f d\sigma, $$ where $d\sigma$ is the
measure defined in Corollary \ref{EM}. \end{lem}

Combining Lemma \ref{DON} and (\ref{scalarsymp2}) we obtain
$$\int_P \frac{1}{2} f(x) S(\mu^{-1}(x)) + \frac{1}{2}
\langle \nabla_{\rho} \mu_h (\mu^{-1}(x), \nabla_x^2 f(x) \rangle
dx  =  \frac{1}{2} \int_{\partial P} f d\sigma, $$ proving that
the two term expansion in Corollary \ref{EM} is correct.

\begin{rem}

\noindent{\bf (i)} We note that in \cite{D2} Lemma 3.3.5, the
boundary term is given the $-$ sign. However, the measure
$d\sigma$ was only defined there (page 307) up to sign. The sign
of this term is universal and by comparing with the
one-dimensional case, we see that it is positive. \medskip

\noindent{\bf (ii)} To connect this calculation to  the classical
one-dimensional case (\ref{1DIM}), and perhaps clarify the
notation, we note that its $N^{m-1}$  (with $m = 1$),
$$ \int_0^1 f(x) dx + \frac{1}{2} \int_0^1 (x -
x^2) f''(x) dx,   $$ may be expressed in terms of the Fubini-Study
\kahler potential and moment map as
$$\int_0^1  \frac{d}{d\rho} \mu_{FS}(\mu^{-1}(x)) f''(x) dx, \;\; x =
\mu(e^{\rho/2}), $$ since
$$\phi_{FS}(e^{\rho/2}) = \log (1 + e^{\rho}), \; \frac{d}{d\rho}
\phi_{FS}(e^{\rho/2}) = \mu_{FS}(e^{\rho/2}) = \frac{e^{\rho}}{1 +
e^{\rho}} = x, $$  and
$$\frac{d^2}{d\rho^2}
\phi_{FS}(e^{\rho/2}) = \frac{e^{\rho}}{(1 + e^{\rho})^2} = x (1 -
x). $$ Regarding $S$, we recall that it is the scalar curvature of
the metric $g_{1 \bar{1}}$ associated to the \kahler form
$\omega_{FS} = \frac{i}{2} \ddbar (1 + |z|^2) $, thus
$$ S = - \frac{\partial^2}{\partial z \partial \bar{z}} \log (1 +
|z|^2)^{-2} = 2 Tr g_{1 \bar{1}} = 2. $$
\end{rem}

\end{document}